\documentclass[a4paper,12pt]{article}
\usepackage{amsmath}
\usepackage{amssymb}
\usepackage{latexsym}
\usepackage{amsthm}
\newtheorem{Thm}{Theorem}[section]

\newtheorem{Lem}[Thm]{Lemma}

\theoremstyle{definition}

\newtheorem{Rem}[Thm]{Remark}
\newtheorem{Exa}[Thm]{Example}

\begin{document}

\title{Singularity results for functional equations driven by linear fractional transformations}
\author{Kazuki Okamura\footnote{Graduate School of Mathematical Sciences, The University of Tokyo, Komaba 3-8-1, Meguro-ku, Tokyo 153-8914, Japan \, e-mail : \texttt{kazukio@ms.u-tokyo.ac.jp}}}
\date{}
\maketitle

\begin{abstract}
We  consider functional equations driven by linear fractional transformations,
which are special cases of de Rham's functional equations.
We consider Hausdorff dimension of the measure whose distribution function is the solution. 
We give a necessary and sufficient condition for singularity. 
We also show that they have a relationship with stationary measures. 
\end{abstract}

\section{Introduction}

De Rham \cite{de} considered the following functional equation. 
\begin{equation}
	f(x) = \begin{cases}
						F_{0}(f(2x)) & 0 \leq x \leq 1/2 \\
						F_{1}(f(2x-1)) & 1/2  \leq x \leq 1. \tag{1.1}
				          \end{cases}
\end{equation}

He showed that there exists a unique, continuous and strictly increasing solution $f$ of (1.1), 
if $F_{0}$ and $F_{1}$ are strictly increasing contractions on $[0,1]$ such that 
$0 = F_{0}(0) < F_{0}(1) = F_{1}(0) < F_{1}(1) = 1$.

Let $\mu_{p}$, $p \in (0,1)$,  be the probability measure on $\{0,1\}$ with $\mu_{p}(\{0\}) = p$ and $\mu_{p}(\{1\}) = 1-p$.
Let $\mu_{p}^{\otimes\mathbb{N}}$ be the infinite product measure of $\mu_{p}$ on $\{0,1\}^{\mathbb{N}}$.
Let $\varphi : \{0,1\}^{\mathbb{N}} \to [0,1]$  be a function defined by $\varphi((x_{n})_{n}) = \sum_{n=1}^{\infty} x_{n}/2^{n}$.
Let $f_{p}$ be the distribution function of the image measure of $\mu_{p}^{\otimes\mathbb{N}}$ by  $\varphi$.
We see that $f_{p}$ is a singular function on $[0,1]$ if $p \ne 1/2$ and $f_{1/2}(x) = x$. 
$f_{p}$ is called Lebesgue's singular function if $p \ne 1/2$.
De Rham \cite{de} studied $f_{p}$ as a solution of (1.1) for $F_{0}(x) = px$ and $F_{1}(x) = (1-p)x + p$. 
This is a typical example of (1.1).  
We treat this case in Example 5.1.

In the above case, 
both $F_{0}$ and $F_{1}$ are  affine maps on $\mathbb{R}$. 
It is natural to consider singularities for the solution of (1.1) for more general $F_{0}, F_{1}$.
However, it is difficult to see singularities for general cases, 
because we do not see that what properties of $F_{0}$ and $F_{1}$ definitely affect  singularities. 
Some recent results concerning singularities are Berg and Kr\"uppel \cite{BK}, Kawamura \cite{Ka}, Kr\"uppel \cite{Kr}, Protasov \cite{Pr}. 
But results for general cases are scarce.

In this paper, 
we consider the equation (1.1) under the assumption that both $F_{0}$ and $F_{1}$ are linear fractional transformations. 
Let $\Phi(A ; z) = \dfrac{az + b}{cz + d}$ for a $2 \times 2$ real matrix $A = \begin{pmatrix} a & b  \\ c & d \end{pmatrix}$ and $z \in \mathbb{R}$.
Let $F_{i}(x) = \Phi(A_{i} ; x)$, $x \in [0,1]$, $i = 0,1$,
such that $2 \times 2$ real matrices $A_{i} = \begin{pmatrix} a_{i} & b_{i}  \\ c_{i} & d_{i} \end{pmatrix}$, $i = 0,1$,  satisfy the following conditions (A1) - (A3).\\
(A1) $0 = b_{0} < \dfrac{a_{0} + b_{0}}{c_{0} + d_{0}} = \dfrac{b_{1}}{d_{1}} < \dfrac{a_{1} + b_{1}}{c_{1} + d_{1}} = 1$.\\
(A2) $a_{i}d_{i} - b_{i}c_{i} > 0$, $i = 0,1$.\\
(A3) $(a_{i}d_{i} - b_{i}c_{i})^{1/2} < \min\{d_{i}, c_{i}+d_{i}\}$, $i = 0,1$.

The conditions (A1) - (A3) guarantee that  $F_{i} := \Phi(A_{i} ; \cdot)$, $i = 0,1$,  satisfy de Rham's conditions.
Let us denote $\mu_{f}$ be the probability measure such that the solution $f$ of (1.1) is the distribution function of $\mu_{f}$.

Let $\alpha = \min\{0, c_{0}/(d_{0}-a_{0}), c_{1}/b_{1}\}$, $\beta = \max \{0, c_{0}/(d_{0}-a_{0}), c_{1}/b_{1} \}$ and $\gamma = 1/\Phi(A_{0} ; 1)  > 1$.
Let $p_{0}(x) = (x+1)/(x+\gamma)$ and $p_{1}(x) = 1 - p_{0}(x)$ for $x > -\gamma$.
Let $s(p) = -p \log p - (1-p) \log (1-p)$ for $p \in [0,1]$.
We denote the $s$-dimensional Hausdorff measure, $s \in (0,1]$, of $E \subset \mathbb{R}$ by  $H_{s}(E)$ and the Hausdorff dimension of $E$ by $\dim_{H}(E)$.

The following theorems are main results in this paper.
 
\begin{Thm}
$(1)$ There exists a Borel set $K_{0}$ such that $\mu_{f}(K_{0}) = 1$ and $\dim_{H}(K_{0}) \leq  \max \{ s(p_{0}(y)) ; y \in [\alpha, \beta] \}/ \log 2$.\\
$(2)$ We have that $\mu_{f}(K) = 0$ for any Borel set $K$ with $\dim_{H}(K) < \min \{ s(p_{0}(y)) ; y \in [\alpha, \beta] \}/\log 2$.
\end{Thm}

\begin{Thm} 
$(1)$ If both 
$(i)$ $(c_{0} + d_{0} - 2a_{0})(d_{0} - a_{0}) =  a_{0}c_{0}$, 
and $(ii)$ $(a_{1} - 2c_{1})(d_{1} - 2b_{1}) =  b_{1}c_{1}$
are satisfied, 
then $\mu_{f}(dx) = (1+2c_{0})/(-2c_{0}x+1+2c_{0})^{2} dx$.
In particular, $\mu_{f}$ is absolutely continuous.\\
$(2)$ If either $(i)$ or $(ii)$ fails, 
then there exists a Borel set $K_{1}$ such that  $\mu_{f}(K_{1}) = 1$ and $\dim_{H}(K_{1}) < 1$.
In particular, $\mu_{f}$ is singular.
\end{Thm}

We remark that singularity is robust as a function of $a_{i}, b_{i}, c_{i}, d_{i}$, $i=0,1$,
on the other hand,
absolute continuity is not robust. 

This paper is organized as follows. 
In section 2, we state some lemmas.
In section 3, we show the main results.
In section 4, we state a relationship between these functional equations and stationary measures.
In section 5, we give examples and remarks.

\textbf{Acknowledgements.} \, The author wishes to express his gratitude to his adviser Professor Shigeo Kusuoka for his encouragement. 
The author also wishes to express his gratitude to the referee for his or her comments.


\section{Lemmas}

First, we introduce some notation.

Let $X_{n} : [0,1) \to \{0,1\}$ , $n \geq 1$ be given by $X_{n}(x) = [2^{n}x] - 2[2^{n-1}x]$, $x \in [0,1)$.
Let $\rho_{n}(i_{1}, \dots , i_{n}) = \mu_{f}(\{X_{j} = i_{j}, 1 \leq j \leq n\})$ for $n \geq 1$, $i_{1}, \dots , i_{n} \in \{0,1\}$
and $R_{n}(x) = \rho_{n}(X_{1}(x), \dots , X_{n}(x))$ for $n \geq 1$ and $x \in [0,1)$. 
Let \\
$I_{n}(x) = \left[\sum_{i=1}^{n} 2^{-j} X_{j}(x), \sum_{i=1}^{n} 2^{-j} X_{j}(x) + 2^{-n}\right) = \left[2^{-n}[2^{n}x], 2^{-n}([2^{n}x]+1)\right)$.
Then, $x \in I_{n}(x)$, $x \in [0,1)$, 
and, $X_{n}(y) = X_{n}(x)$ and $I_{n}(y) = I_{n}(x)$ for $y \in I_{n}(x)$.
We have that 
$R_{n}(x) = \mu_{f}\left(\left\{X_{j} = X_{j}(x), 1 \leq j \leq n\right\}\right) = \mu_{f}\left(I_{n}(x)\right)$.
Let \[ \begin{pmatrix} p_{n}(x) & q_{n}(x)  \\ r_{n}(x) & s_{n}(x) \end{pmatrix}  = A_{X_{1}(x)} \cdots A_{X_{n}(x)}, x \in [0,1). \]

\begin{Lem}
Let $n \geq 1$ and $i_{1}, \dots, i_{n} \in \{0,1\}$.
Then we have the following.\\ 
$(1)$ $f\left(\sum_{i=1}^{n} 2^{-j} i_{j}\right) = \Phi(A_{i_{1}} \cdots A_{i_{n}}; 0)$ and 
$f\left(\sum_{i=1}^{n} 2^{-j} i_{j} + 2^{-n}\right) = \Phi(A_{i_{1}} \cdots A_{i_{n}}; 1)$.\\
$(2)$ $R_{n+1}(x)/R_{n}(x) = p_{X_{n+1}(x)}(r_{n}(x)/s_{n}(x))$.
\end{Lem}

\begin{proof}
(1) By recalling (1.1), we easily show the assertion by induction in $n$.

(2) By the assertion (1), we have that  
\begin{align*}
R_{k}(x) 
&= \Phi(A_{X_{1}(x)} \cdots A_{X_{k}(x)}; 1) - \Phi(A_{X_{1}(x)} \cdots A_{X_{k}(x)}; 0) \\
&= \frac{p_{k}(x)s_{k}(x)-q_{k}(x)r_{k}(x)}{s_{k}(x)(r_{k}(x)+s_{k}(x))}.
\end{align*}

By computation, 
we have that
\begin{multline*} \frac{R_{n+1}(x)}{R_{n}(x)} 
= \frac{(\det A_{X_{n+1}(x)}) s_{n}(x)}{b_{X_{n+1}(x)}r_{n}(x) + d_{X_{n+1}(x)}s_{n}(x)} \\
\times \frac{r_{n}(x)+s_{n}(x)}{\left(a_{X_{n+1}(x)} + b_{X_{n+1}(x)}\right)r_{n}(x) + \left(c_{X_{n+1}(x)} + d_{X_{n+1}(x)}\right)s_{n}(x)}. 
\end{multline*}

By noting (A2), we have that
\[ \frac{R_{n+1}(x)}{R_{n}(x)} = p_{X_{n+1}(x)}\left(\frac{r_{n}(x)}{s_{n}(x)}\right). \]


Thus we obtain the assertion (2).
\end{proof}

Now we state some properties of $\Phi({}^t \! A_{i};\cdot)$, $i=0,1$.

We remark that $\Phi({}^t \! A_{0};\cdot)$ (resp. $\Phi({}^t \! A_{1};\cdot)$) is well-defined and continuous on $\mathbb{R}$ (resp. $(-\gamma, \infty)$).

\begin{Lem}
$(1)$ $d_{0} > a_{0} > 0$, $b_{1}+c_{1} > 0$ and $\alpha > -1$.\\
$(2)$ $\Phi({}^t \! A_{0}; z) = z$ if and only if $z = c_{0}/(d_{0}-a_{0})$. \\
$(3)$ $\Phi({}^t \! A_{1}; z) = z$ if and only if $z = -1$ or $z = c_{1}/b_{1}$.
\end{Lem}

\begin{proof}
(1) By (A2) and (A3), we have that $d_{0} > 0$, and then $a_{0} > 0$.
By (A3) and (A1), we have that $0 < (a_{0}d_{0})^{1/2}  = (a_{0}d_{0}-b_{0}c_{0})^{1/2} < d_{0}$ and then $0 < a_{0} < d_{0}$.

By (A1), we have that $a_{1}+b_{1} = c_{1}+d_{1}$ and then $a_{1}d_{1}-b_{1}c_{1} = (c_{1}+d_{1})(d_{1}-b_{1})$.
By (A2) and (A3), we have that $c_{1}+d_{1}>0$, and then $d_{1}-b_{1}>0$. 
By (A3), we have that $0 < (c_{1}+d_{1})^{1/2}(d_{1}-b_{1})^{1/2} < c_{1}+d_{1}$.
Hence we have that $d_{1}-b_{1} < c_{1}+d_{1}$, and then $b_{1}+c_{1}>0$.

By (A2) and (A3), we have that $d_{1}>0$.
By (A1), we have that $b_{1}>0$.
Since $b_{1}+c_{1}>0$, we see that $c_{1}/b_{1} > -1$.
Then, we have that $c_{0}/(d_{0}-a_{0}) > -1$ by noting (A1) and $a_{0} < d_{0}$.
Now we have that $a = \min \{0, c_{0}/(d_{0}-a_{0}), c_{1}/b_{1}\} > -1$.

(2) Since $b_{0} = 0$, we have that $\Phi({}^t \! A_{0}; z) - z = -(d_{0}-a_{0})z/d_{0} + c_{0}/d_{0}$.
Since $d_{0} > a_{0}$, we see that $\Phi({}^t \! A_{0}; z) = z$ if and only if $z = c_{0}/(d_{0}-a_{0})$.

(3) Since 
\begin{align*} 
\Phi({}^t \! A_{1} ; z) - z &= \frac{-b_{1}z^{2}-(d_{1}-a_{1})z+c_{1}}{b_{1}z+d_{1}} \\
&= \frac{(-b_{1}z+c_{1})(z+1)}{b_{1}z+d_{1}} = -\frac{(z+1)(z- c_{1}/b_{1})}{z+\gamma}, 
\end{align*}
we see that $\Phi({}^t \! A_{1}; z) = z$ if and only if $z = -1$ or $z = c_{1}/b_{1}$.
\end{proof}


Let $\mathcal{F}_{n} = \sigma(X_{1}, \dots, X_{n})$, $n \geq 1$. 
Let $L_{n} = \sum_{i = 1}^{n} E^{\mu_{f}}[- \log (R_{i}/R_{i-1}) | \mathcal{F}_{i-1}]$ and $M_{n} = - \log R_{n} - L_{n}$, $n \geq 1$. 
Then we have the following.
\begin{Lem} 
$(1)$ $L_{n+1}(x) - L_{n}(x) = s(p_{0}(r_{n}(x)/s_{n}(x)))$ for $\mu_{f}$-a.s.$x \in [0,1)$.\\
$(2)$ $M_{n}/n \to 0, \, (n \to \infty)$ for $\mu_{f}$-a.s.
\end{Lem}

\begin{proof}
(1) It is sufficient to show that for any $x \in [0,1)$, 
\[ \int_{I_{n}(x)} s\left(p_{0}\left(\frac{r_{n}(y)}{s_{n}(y)}\right)\right) \mu_{f}(dy) =  \int_{I_{n}(x)} -\log \left(\frac{R_{n+1}(y)}{R_{n}(y)}\right) \mu_{f}(dy). \]

Since $r_{n}(y)/s_{n}(y) = r_{n}(x)/s_{n}(x)$ for $y \in I_{n}(x)$, we see that
\[ \int_{I_{n}(x)} s\left(p_{0}\left(\frac{r_{n}(y)}{s_{n}(y)}\right)\right) \mu_{f}(dy) =  \mu_{f}(I_{n}(x))  s\left(p_{0}\left(\frac{r_{n}(x)}{s_{n}(x)}\right)\right). \]

By Lemma 2.1(2), we see that 
\[ - \log \frac{\mu_{f}(I_{n+1}(y))}{\mu_{f}(I_{n}(y))} = - \log \frac{R_{n+1}(y)}{R_{n}(y)} = -\log p_{X_{n+1}(y)}\left(\frac{r_{n}(x)}{s_{n}(x)}\right), \]
and, 
\[ \int_{I_{n}(x)} -\log \left(\frac{R_{n+1}(y)}{R_{n}(y)}\right) \mu_{f}(dy) = \int_{I_{n}(x)} -\log \left(p_{X_{n+1}(y)}\left(\frac{r_{n}(y)}{s_{n}(y)}\right)\right) \mu_{f}(dy) \]
\[ = - \mu_{f}\left(I_{n}(x) \cap \{X_{n+1} = 0\}\right) \log p_{0}\left(\frac{r_{n}(x)}{s_{n}(x)}\right) - \mu_{f}\left(I_{n}(x) \cap \{X_{n+1} = 1\}\right) \log p_{1}\left(\frac{r_{n}(x)}{s_{n}(x)}\right)  \]
\[ = \mu_{f}\left(I_{n}(x)\right) s\left(p_{0}(r_{n}(x)/s_{n}(x))\right),\]
which implies the assertion (1).

(2)  By noting Jensen's inequality, we have that 
\begin{align*}
E^{\mu_{f}}\left[(M_{k}-M_{k-1})^{2}\right] &\leq 2  \left(E^{\mu_{f}}\left[(- \log R_{k}+ \log R_{k-1})^{2}\right]+E^{\mu_{f}}\left[(L_{k} - L_{k-1})^{2}\right] \right)\\
&\leq 4 E^{\mu_{f}}\left[(- \log R_{k}+ \log R_{k-1})^{2}\right].
\end{align*}

Let $C_{0} = \sup\left\{ x(\log x)^{2} + (1-x)(\log (1-x))^{2} : x \in [0,1]\right\} < +\infty$.
We will show that $E^{\mu_{f}}\left[(\log (R_{n+1}/R_{n}))^{2}\right] \leq  C_{0}$ for any $n \geq 1$. 

Let $\tau(p) = p (\log p)^{2} + (1-p) (\log (1-p))^{2}$ for $p \in [0,1]$.
We remark that $\tau(p) = \tau(1-p)$. 
Then we have that
\[ E^{\mu_{f}}\left[(- \log R_{n}+ \log R_{n-1})^{2}\right]  = \sum_{k=0}^{2^{n}-1}  \mu_{f}\left(I_{n}\left(\frac{k}{2^{n}}\right)\right) \left(\log \frac{R_{n}(k/2^{n})}{R_{n-1}(k/2^{n})}\right)^{2} \]
\begin{multline*} = \sum_{k=0}^{2^{n-1}-1} \mu_{f}\left(I_{n}\left(\frac{2k}{2^{n}}\right)\right) \left(\log \frac{R_{n}(2k/2^{n})}{R_{n-1}(2k/2^{n})}\right)^{2} \\
+ \mu_{f}\left(I_{n}\left(\frac{2k+1}{2^{n}}\right)\right) \left(\log \frac{R_{n}(2k+1/2^{n})}{R_{n-1}(2k+1/2^{n})}\right)^{2}.
\end{multline*}

By noting that $R_{n-1}(2k/2^{n}) = R_{n-1}(2k+1/2^{n}) = R_{n-1}(k/2^{n-1})$, $\mu_{f}\left(I_{n}(2k/2^{n})\right) = R_{n}(2k/2^{n})$ and $\mu_{f}\left(I_{n}(2k+1/2^{n})\right) = R_{n}(2k+1/2^{n})$, 
we have that
\[ E^{\mu_{f}}\left[ \left(\log \frac{R_{n}}{R_{n-1}}\right)^{2}\right] = \sum_{k=0}^{2^{n-1}-1} R_{n-1}\left(\frac{k}{2^{n-1}}\right) \tau \left( \frac{R_{n}(k/2^{n-1})}{R_{n-1}(k/2^{n-1})}\right)  \le C_{0}.\]

Thus we have that $\sup_{k \geq 1} E^{\mu_{f}}[(M_{k}-M_{k-1})^{2}] \leq 4C_{0} < +\infty$.
Since $\{M_{n}\}$ is an $\{\mathcal{F}_{n}\}$-martingale, 
$\{M_{n}^{2}\}$ is an $\{\mathcal{F}_{n}\}$-submartingale.
Noting that $M_{0} = 0$, we have that 
$E^{\mu_{f}}[M_{n}^{2}] = \sum_{k=1}^{n} E^{\mu_{f}}[(M_{k}-M_{k-1})^{2}]$.

By Doob's submartingale inequality, 
we have that
\[ \mu_{f}\left( \max_{1 \leq k \leq 2^{l}} M_{k}^{2} \geq \epsilon 4^{l}\right) 
\leq \frac{E^{\mu_{f}}[M_{2^{l}}^{2}]}{\epsilon 4^{l}} 
\leq \frac{4C_{0}}{\epsilon 2^{l}}, \, l \geq 1, \, \epsilon > 0. \]

Now we have that for $\mu_{f}$-a.s.$x$,
there exists $m = m(x)\in \mathbb{N}$ such that \\
$\max_{1 \leq k \leq 2^{l}} (M_{k}(x)/2^{l})^{2} \leq \epsilon$, $l \geq m$, and then, 
$(M_{n}(x)/n)^{2} \leq 4\epsilon$, $n \geq 2^{m}$.
Then we see that $\limsup_{n \to \infty} (M_{n}/n)^{2} \leq \epsilon$, $\mu_{f}$-a.s.,  which implies our assertion.
\end{proof}

\begin{Lem}
$(1)$ Suppose that  $\limsup_{n \to +\infty} (-\log R_{n})/n \leq \theta_{1}$ for a constant $\theta_{1}$, 
then there exists a Borel set $K_{0}$ such that $\mu_{f}(K_{0}) = 1$ and $\dim_{H}(K_{0}) \leq  \theta_{1}/\log 2$.\\
$(2)$ Suppose that $\liminf_{n \to +\infty} (-\log R_{n})/n \geq \theta_{2}$ for a constant $\theta_{2}$, 
then we have that $\mu_{f}(K) = 0$ for any Borel set $K$ with $\dim_{H}(K) < \theta_{2}/\log 2$.
\end{Lem}

\begin{proof}
We denote the diameter of a set $G \subset \mathbb{R}$ by $\textrm{diam}(G)$.

(1)  Let $Y_{\epsilon, n} = \bigcap_{k \geq n} \left\{ (-\log R_{k})/k \leq \theta_{1} + \epsilon\right\}$.
Then we have that $\mu_{f}\left(\bigcup_{n \geq 1} Y_{\epsilon, n}\right) = 1$.
Let $\mathcal{A}_{\epsilon, k}$ be the set of $I_{k}(x)$, $x \in [0,1)$, such that $R_{k}(x) \geq \exp(-k(\theta_{1}+\epsilon))$.
Then, for any $k \geq n$, 
$\{I_{k}(x) \in \mathcal{A}_{\epsilon, k} : x \in Y_{\epsilon, n}\}$ is a $2^{-k}$-covering of $Y_{\epsilon, n}$. 

Since $\mu_{f}([0,1)) = 1$, 
we see that $\sharp(\mathcal{A}_{\epsilon, k}) \exp(-k(\theta_{1}+\epsilon)) \leq 1$.
Then
\[ \sum_{I \in \mathcal{A}_{\epsilon, k}} \textrm{diam}(I)^{(\theta_{1}+2\epsilon)/\log 2} = \sharp(\mathcal{A}_{\epsilon, k}) \exp \left(-k(\theta_{1}+2\epsilon)\right) \leq \exp (-k\epsilon). \]
By letting $k \to +\infty$, we see $H_{(\theta_{1}+2\epsilon)/\log 2}(Y_{\epsilon, n}) = 0$.

Let $K_{0} = \bigcap_{k \geq 1} \bigcup_{n \geq 1} Y_{1/k, n}$. 
Then, we have that $\mu_{f}(K_{0}) = 1$ and $H_{(\theta_{1}+2\epsilon)/\log 2}(K_{0}) = 0$ for any $\epsilon > 0$.
Hence $\dim_{H}(K_{0}) \leq  \theta_{1}/\log 2$.

(2) Let $K$ be a Borel set such that $\dim_{H}(K) < \theta_{2}/\log 2$.
Then, there exists $\epsilon > 0$ such that $H_{(\theta_{2}-\epsilon)/\log 2}(K) = 0$.
Then, for any $n \geq 1$ and $\delta > 0$, there exist intervals  $\{U(n,l)\}_{l = 1}^{\infty}$ on $[0,1)$ 
such that $K \subset \bigcup_{l \geq 1} U(n,l)$ and $\textrm{diam}(U(n,l)) < 2^{-n}$ for $l \geq 1$ and 
$\sum_{l \geq 1} \textrm{diam}(U(n,l))^{(\theta_{2}-\epsilon)/\log 2} \leq \delta$.
For each $l \geq 1$, let $k(n,l) > n$ be the integer such that $2^{-k(n,l)} \leq \textrm{diam}(U(n,l)) < 2^{-(k(n,l)-1)}$.

Let $Z_{\epsilon, n} = \bigcap_{k \geq n} \left\{(-\log R_{k})/k \geq \theta_{2} - \epsilon\right\}$.
Then we have that \\
$\lim_{n \to \infty} \mu_{f}(Z_{\epsilon, n}) = \mu_{f}\left(\bigcup_{n \geq 1} Z_{\epsilon, n}\right) = 1$, 
and, \[ \mu_{f}\left(I_{k(n,l)}(y)\right) =  R_{k(n,l)}(y) \leq \exp \left(-k(n,l)(\theta_{2}-\epsilon)\right) \leq \textrm{diam}(U(n,l))^{(\theta_{2}-\epsilon)/\log 2}, \] for $y \in Z_{\epsilon, n}$ and $l \geq 1$.

Since $\textrm{diam}(I_{k(n,l)}(x)) = 2^{-k(n,l)}$ and $\textrm{diam}(U(n,l)) < 2^{-(k(n,l)-1)}$, we see that 
$\sharp \left\{I_{k(n,l)}(x) ;  I_{k(n,l)}(x) \cap U(n,l)  \neq \emptyset \right\} \leq 3$ and that \\
$\mu_{f}\left(K \cap Z_{\epsilon,n} \cap U(n,l)\right) \leq 3 \textrm{diam}(U(n,l))^{(\theta_{2}-\epsilon)/\log 2}$.

Noting that $K \subset \bigcup_{l \geq 1} U(n,l)$, we see that 
\[ \mu_{f}(K \cap Z_{\epsilon,n}) \leq  \sum_{l \geq 1}\mu_{f}(K \cap Z_{\epsilon,n} \cap U(n,l)) \leq 3 \sum_{l \geq 1} \textrm{diam}(U(n,l))^{(\theta_{2}-\epsilon)/\log 2} \leq 3 \delta.\]

Since $\delta$ is taken arbitrarily, we see that $\mu_{f}(K \cap Z_{\epsilon,n}) = 0$.
Recalling $\mu_{f}\left(\bigcup_{n \geq 1} Z_{\epsilon, n}\right) = 1$, we see that $\mu_{f}(K) = 0$.
\end{proof}


\section{Proofs of Main Theorems}

\begin{Lem}
Let $n \geq 1$ and $i_{1}, \dots, i_{n} \in \{0,1\}$. 
Then, \\
$\alpha \leq \Phi({}^t \! A_{i_{n}} \cdots {}^t \! A_{i_{1}} ; \alpha) \leq \Phi ({}^t \! A_{i_{n}} \cdots {}^t \! A_{i_{1}} ; \beta) \leq \beta$. \\
In particular,
$r_{n}(x)/s_{n}(x) \in [\alpha, \beta]$ for $n \geq 1$ and $x \in [0,1)$.
\end{Lem}

\begin{proof}
By noting Lemma 2.2, we have that $\Phi({}^t \! A_{0} ; z) - z = -(d_{0}-a_{0})z/d_{0} + c_{0}/d_{0}$ and $\Phi({}^t \! A_{1}; z) - z = -(z+1)(z- c_{1}/b_{1})/(z+\gamma)$.
We remark that $\alpha > -1 > -\gamma$.
Since $\alpha \leq c_{0}/(d_{0}-a_{0}),  c_{1}/b_{1} \leq \beta$, we see that $\alpha \leq \Phi({}^t \! A_{i}; \alpha) \leq \Phi({}^t \! A_{i}; \beta) \leq \beta$ for $i = 0,1$.

Since $\Phi({}^t \! A_{0}; \cdot)$ and $\Phi({}^t \! A_{1}; \cdot)$ are increasing, we obtain the assertion by induction in $n$.

We have that $\alpha \leq 0 \leq \beta$ by the definition of $\alpha$ and $\beta$.
Since $r_{n}(x)/s_{n}(x) = \Phi ({}^t \! A_{X_{n}(x)} \cdots {}^t \! A_{X_{1}(x)} ; 0)$, we see that $r_{n}(x)/s_{n}(x) \in [\alpha, \beta]$.  
\end{proof}

Now we show Theorem 1.1.

By noting Lemma 2.3 and Lemma 3.1, we see that for $\mu_{f}$-a.s., 
\begin{align*} 
\limsup_{n \to +\infty} \frac{-\log R_{n}}{n} = \limsup_{n \to \infty} \frac{L_{n}}{n} &=\limsup_{N \to \infty} \frac{1}{N} \sum_{n=1}^{N} s\left(p_{0}\left(\frac{r_{n}(x)}{s_{n}(x)}\right)\right)  \\
&\le \max \left\{ s(p_{0}(y)) ; y \in [\alpha, \beta] \right\}
\end{align*} 
, and, 
\begin{align*} 
\liminf_{n \to +\infty} \frac{-\log R_{n}}{n} = \liminf_{n \to \infty} \frac{L_{n}}{n} &=\liminf_{N \to \infty} \frac{1}{N} \sum_{n=1}^{N} s\left(p_{0}\left(\frac{r_{n}(x)}{s_{n}(x)}\right)\right)\\
&\ge \min \left\{s(p_{0}(y)) ; y \in [\alpha, \beta]\right\}.
\end{align*}

Let $\theta_{1} = \max \left\{ s(p_{0}(y)) ; y \in [\alpha, \beta] \right\}$ 
and $\theta_{2} = \min \left\{ s(p_{0}(y)) ; y \in [\alpha, \beta] \right\}$.
Then, by Lemma 2.4(1) (resp. (2)), we obtain the assertion (1) (resp. (2)).

These complete the proof of Theorem 1.1.

\begin{Lem}
Let $\mathbb{N}_{i}(x) = \{n \in \mathbb{N} : X_{n}(x) = i\}$ for $x \in [0,1)$, $i = 0,1$.
Then,  \[ \liminf_{N \to \infty} \frac{|\mathbb{N}_{0}(x) \cap \{1, \dots, N\}|}{N} \geq p_{0}(\alpha) > 0, \, \mu_{f}\textrm{-}a.s.x.\]
\end{Lem}

\begin{proof}
Let $\zeta_{N}(x) = |\mathbb{N}_{0}(x) \cap \{1, \dots, N\}|$.
Then, $\zeta_{N}(x) = \sum_{n=1}^{N} 1_{\{0\}}(X_{n}(x))$.
Let $M_{n} = \sum_{i=1}^{n}\left(1_{\{0\}}(X_{n})-p_{0}(\alpha) \right)$.
Then,  $\{M_{n}\}$ is an $\{\mathcal{F}_{n}\}$-submartingale because  
\[ E^{\mu_{f}}[M_{n+1}-M_{n} | \mathcal{F}_{n}] (x) = E^{\mu_{f}}[1_{\{0\}}(X_{n+1})-p_{0}(\alpha) | \mathcal{F}_{n}] (x) = p_{0}\left(\frac{r_{n}(x)}{s_{n}(x)}\right) - p_{0}(\alpha) \geq 0. \]

We remark that $|M_{n+1}-M_{n}| = |1_{\{0\}}(X_{n+1})-p_{0}(\alpha) | \leq 1+p_{0}(\alpha) $ for $\mu_{f}$-a.s..
By Azuma's inequality \cite{A}, we see that for $N \in \mathbb{N}$ and $0 < c < 1$, 
\[ \mu_{f}(\zeta_{N}< N c p_{0}(\alpha))  = \mu_{f}(M_{N}<-N (1-c) p_{0}(\alpha)) \leq \exp\left(-\frac{N(1-c)^{2}p_{0}(\alpha) ^{2}}{2(1+p_{0}(\alpha) )^{2}}\right). \]
Hence, for any $0 < c < 1$, $\liminf_{N \to \infty} \zeta_{N}/N \geq c p_{0}(\alpha)$ 
for $\mu_{f}$-a.s..
Thus we obtain the assertion.
\end{proof}

\begin{Lem}
We assume that the condition $(i)$ in Theorem 1.2 fails.
Then, \\
$(1)$ There exists $\epsilon_{0} \in (0, 2(\gamma-1))$ 
such that for any $z \in \mathbb{R}$ with $|z-(\gamma-2)| \leq  \epsilon_{0}$,
 $|\Phi({}^t \! A_{0}; z) - (\gamma-2)| > \epsilon_{0}$.

Let $A(x) = \left\{ n \in \mathbb{N} : \left|r_{n}(x)/s_{n}(x) - (\gamma-2)\right| \leq \epsilon_{0}\right\}$, $B(x) = \mathbb{N} \setminus A(x)$, $C(x) = \{n \in A(x) : n-1 \in B(x)\}$ and  $D(x) = B(x) \cup C(x)$. Then we have the following. \\
$(2)$ $\mathbb{N}_{0}(x) \subset D(x)$ for  $x \in [0,1)$.\\
$(3)$ $\liminf_{N \to \infty} |B(x) \cap \{1, \dots, N\}|/N \geq p_{0}(\alpha) /2$, $\mu_{f}$-a.s.$x$.\\
$(4)$ Let $e_{0} = s(p_{0}(\gamma-2+\epsilon_{0})) < \log 2$. Then, 
\[ \limsup_{N \to \infty} \frac{1}{N} \sum_{n=1}^{N} s\left(p_{0}\left(\frac{r_{n}(x)}{s_{n}(x)}\right)\right) \leq \log 2 - \frac{(\log 2 - e_{0})p_{0}(\alpha)}{2}, \, \, \mu_{f} \textrm{-a.s.} x.\]
\end{Lem}

\begin{proof}
(1) This is a direct consequence of the assumption that the condition (i) in Theorem 1.2 fails, that is, $\Phi({}^t \! A_{0}; \gamma-2) \neq \gamma-2$.

(2) It is sufficient to show that $\mathbb{N} \setminus D(x) \subset \mathbb{N}_{1}(x)$.
We see that $\mathbb{N} \setminus D(x) = A(x) \cap (\mathbb{N} \setminus C(x)) = \{n \in A(x) : n-1 \in A(x)\}$. 
We assume that there exists $n \in \mathbb{N} \setminus D(x)$ such that $n \in \mathbb{N}_{0}(x)$.
Since $n-1 \in A(x)$, we have that $|r_{n-1}(x)/s_{n-1}(x) - (\gamma-2)| \leq \epsilon_{0}$.
Since $n \in \mathbb{N}_{0}(x)$, $r_{n}(x)/s_{n}(x) = \Phi({}^t \! A_{0}; r_{n-1}(x)/s_{n-1}(x))$.
By the assertion (1), 
we see that $|r_{n}(x)/s_{n}(x) - (\gamma-2)| > \epsilon_{0}$.
But this is contradict to $n \in A(x)$.

(3) 
By the assertion (2), we see that $|\mathbb{N}_{0}(x) \cap \{1, \dots,  N \}| \leq |D(x) \cap \{1, \dots, N\}|$.
We have that $|C(x) \cap \{1, \dots, N\}| \leq |B(x) \cap \{1, \dots, N\}|$ for any $N \geq 1$,
by the injectivity of the map $h : C(x) \to B(x)$ given by $h(n) = n-1$.
Then we see that $|D(x) \cap \{1, \dots, N\}| \leq 2 |B(x) \cap \{1, \dots, N\}|$, and then, $|\mathbb{N}_{0}(x) \cap \{1, \dots,  N \}| \leq 2 |B(x) \cap \{1, \dots, N\}|$, 
for any $N \geq 1$.

By Lemma 3.2, 
\[ \liminf_{N \to \infty} \frac{|B(x) \cap \{1, \dots, N\}|}{N} \geq \frac{p_{0}(\alpha)}{2}, \mu_{f}\textrm{-}a.s.x.\]
Thus we obtain the assertion (3).

(4) 
By noting the definition of $B(x)$, 
we see that \\
$s(p_{0}(r_{n}(x)/s_{n}(x))) < \max \left\{s(p_{0}(\gamma-2-\epsilon_{0})), s(p_{0}(\gamma-2+\epsilon_{0}))\right\} = e_{0}$ for any $x \in [0,1)$ and $n \in B(x)$.

Now we have that 
\[ \frac{1}{N} \sum_{n=1}^{N} s\left(p_{0}\left(\frac{r_{n}(x)}{s_{n}(x)}\right)\right) = \frac{1}{N} \left(\sum_{n \in A(x),  n \leq N}  + \sum_{n \in B(x),  n \leq N}\right) s\left(p_{0}\left(\frac{r_{n}(x)}{s_{n}(x)}\right)\right). \]

Let $\xi_{N}(x) = |B(x) \cap \{1, \dots, N\}|/N$. Then, 
by noting that $s(p_{0}(r_{n}(x)/s_{n}(x))) \leq \log 2$, we see that 
\[\frac{1}{N} \sum_{n \in A(x),  n \leq N} s\left(p_{0}\left(\frac{r_{n}(x)}{s_{n}(x)}\right)\right) \leq \frac{|A(x) \cap \{1, \dots, N\}|}{N}\log 2 = (1-\xi_{N}(x)) \log 2.\]
Now we have that 
\[ \frac{1}{N} \sum_{n \in B(x), n \leq  N} s\left(p_{0}\left(\frac{r_{n}(x)}{s_{n}(x)}\right)\right) \leq \xi_{N}(x)e_{0}.\]
By noting that $e_{0} < \log2$, 
we see that  
\[ \limsup_{N \to \infty} \left((1-\xi_{N}(x))\log2 + \xi_{N}(x)e_{0}\right) \leq \log 2- (\log2-e_{0}) \liminf_{N \to \infty} \xi_{N}(x).\]
By the assertion (3), we see that $\liminf_{N \to \infty} \xi_{N}(x) \geq p_{0}(\alpha)/2 > 0$ for $\mu_{f}$-a.s.$x$.
Thus we obtain the assertion (4).
\end{proof}

Now we show Theorem 1.2 (1).
We remark that $\Phi(cA;z) = \Phi(A;z)$ for any constant $c > 0$ and  the conditions (A1) - (A3) remain valid for $(cA_{0}, cA_{1})$.
Then, we can assume that $d_{0} = 1$ and $b_{1} = 1$. 

By computation, we see that  \[  A_{0} = \begin{pmatrix} 1/2 & 0  \\ c_{0} & 1 \end{pmatrix}, \, \, \,  A_{1} = \begin{pmatrix}  4c_{0} + 1 &  1 \\  2c_{0} & 2(1 + c_{0}) \end{pmatrix}, \]
and $f(x) = \dfrac{x}{-2c_{0}x + 1 + 2c_{0}}$ satisfies the equation (1.1).
This completes the proof of Theorem 1.2 (1).

\vspace{1\baselineskip}

Now we show Theorem 1.2 (2).
We assume that the condition (i) fails.
Then, by  Lemma 2.3, we have that for $\mu_{f}$-a.s.$x$, 
\[ \limsup_{N \to +\infty} \frac{-\log R_{N}(x)}{N} = \limsup_{N \to \infty} \frac{L_{N}(x)}{N} =\limsup_{N \to \infty} \frac{1}{N} \sum_{n=1}^{N} s\left(p_{0}\left(\frac{r_{n}(x)}{s_{n}(x)}\right)\right). \] 
Then, by noting Lemma 3.3(4) and  Lemma 2.4(1), we obtain the desired result.

We can show the assertion in the same manner  if the condition (ii) fails.

These complete the proof of Theorem 1.2(2).

\section{A relationship with stationary measures}

In this section, we state a relationship between a certain class of de Rham's functional equations and stationary measures. 

We state a general setting.
Let $G$ be a semigroup and $\mu$ be a probability measure on $G$. 
Let $M$ be a topological space. 
We assume that $G$ acts on $M$ measurably, that is, there is a map from $(g,x) \in G \times M$ to $g \cdot x \in M$ satisfying the following conditions :  \\
(1) $(g_{1}g_{2}) \cdot x = g_{1} \cdot (g_{2} \cdot x)$ for any $g_{1}, g_{2} \in G$ and $x \in M$.\\
(2) $x \mapsto g \cdot x$ is measurable map on $M$ for any $g \in G$.

We say that a probability measure $\nu$ on $M$ is a $\mu$\textit{-stationary measure} if 
\begin{equation}
 \nu(B) = \int_{G} \nu(h^{-1}B) \mu(dh), \tag{4.1}
\end{equation} 
for any $B \in \mathcal{B}(M)$.
Furstenberg \cite{Fu} Lemma 1.2 showed that if $M$ is a compact metric space, then there exists a $\mu$-stationary measure.

Let \[ G = \left\{ \begin{pmatrix} a & b \\ c & d  \end{pmatrix} \in M(2; \mathbb{R}) : ad > bc, b \geq 0, d > 0,  0 < a + b \leq c + d \right\}, \]
and, $M = [0,1]$.
Then $G$ is a semigroup.  
We define a continuous action of $G$ to $M$ by $A\cdot z =  \Phi \left(A ; z\right)$.
For $(A_{0}, A_{1})$ satisfying (A1)-(A3), we see that $A_{0}, A_{1} \in G$. 
Let $\mu$ be a probability measure on $G$ such that $\mu(\{A_{0}\}) = \mu(\{A_{1}\}) = 1/2$. 
Then we have the following.

\begin{Lem}
$(1)$ For $k \geq 1$, 
\begin{equation*}
\begin{cases} A_{0}^{-1}(f(I_{k}(x))) = f(I_{k-1}(2x)),  \, \, A_{1}^{-1}(f(I_{k}(x))) = \emptyset  & x \in [0, 1/2) \\
 A_{0}^{-1}(f(I_{k}(x))) = \emptyset, \, \,  A_{1}^{-1}(f(I_{k}(x))) = f(I_{k-1}(2x-1))  & x \in [1/2, 1). 
\end{cases}
\end{equation*}\\
$(2)$ For any $\mu$-stationary measure $\nu$ and $k \geq 1$,\\ 
\begin{equation*}
\nu(f(I_{k}(x))) = \begin{cases} \nu(f(I_{k-1}(2x)))/2 & x \in [0, 1/2) \\
\nu(f(I_{k-1}(2x-1)))/2 & x \in [1/2, 1). 
\end{cases}
\end{equation*}\\
$(3)$ There exists exactly one $\mu$-stationary measure $\nu$.
\end{Lem}

\begin{proof}
(1) By Lemma 2.1(1), we see that \\
$f(I_{k}(x))= \Phi(A_{X_{1}(x)} \cdots A_{X_{k}(x)};  [0,1)) = \Phi(A_{X_{1}(x)} ; \Phi(A_{X_{2}(x)} \cdots A_{X_{k}(x)};  [0,1)))$.
We see that $f(I_{k-1}(2x)) = \Phi(A_{X_{2}(x)} \cdots A_{X_{k}(x)};  [0,1)) = A_{0}^{-1}(f(I_{k}(x)))$, $x \in [0, 1/2)$,
 and, $f(I_{k-1}(2x-1)) = \Phi(A_{X_{2}(x)} \cdots A_{X_{k}(x)};  [0,1)) = A_{1}^{-1}(f(I_{k}(x)))$, $x \in [1/2, 1)$.
Since $\Phi(A_{0} ; [0,1)) \cap \Phi(A_{1} ; [0,1)) = \emptyset$, 
$A_{1}^{-1}(f(I_{k}(x))) = \emptyset$, $x \in [0, 1/2)$,  and,  
$A_{0}^{-1}(f(I_{k}(x))) = \emptyset$, $x \in [1/2, 1)$.
Thus we have the assertion (1).

(2) By noting the assertion (1) and (4.1), we obtain the desired result.

(3) Let $\nu_{i}$, $i=0,1$,  be two $\mu$-stationary measures.
By the assertion (2), 
we see that $\nu_{0}(f(I_{k}(x))) = \nu_{1}(f(I_{k}(x)))$ for $k \geq 1$, $x \in [0,1)$.
Let \\
$\mathcal{C} = \left\{f(\sum_{i=1}^{k} 2^{-j} X_{j}(x)) : k \geq 1, x \in [0,1)\right\} = \left\{f(l/2^{k}) : 0 \leq l \leq 2^{k-1}, k \geq 1\right\}$.
Then, we have that $\nu_{0}([a,b)) = \nu_{1}([a,b))$ for $a, b \in \mathcal{C}$.
Since $f$ is continuous on $[0,1]$, 
$\mathcal{C}$ is dense in $[0,1]$.
Thus we see that $\nu_{0} = \nu_{1}$.
\end{proof}

\begin{Lem}
Let $g : [0,1] \to [0,1]$ be the inverse function of the solution $f$ of $(1.1)$. 
Then, \\
$(1)$ $g$ is continuous and strictly increasing. 
Hence, $\mu_{g}$ is well-defined.\\
$(2)$ $\mu_{f}$ is singular if and only if  $\mu_{g}$ is so.
\end{Lem}

\begin{proof}
(1) Noting that $f$ is continuous and strictly increasing on $[0,1]$, $f(0) = 0$ and $f(1) = 1$, 
we obtain the desired result.

(2) Since $l([a,b)) = \mu_{f}\left(f^{-1}([a,b))\right) =  \mu_{g}\left(g^{-1}([a,b))\right)$ for $0 \leq a \leq b \leq 1$, we see that 
$l(B) = \mu_{f}\left(f^{-1}(B)\right) =  \mu_{g}\left(g^{-1}(B)\right)$ for any Borel set $B$.

We assume that $\mu_{f}$ is singular.
Then, there exists a Borel set $B_{0}$ such that $\mu_{f}(B_{0}) = 0$ and $l(B_{0}) = 1$.
Then, $\mu_{g}\left(g^{-1}(B_{0})\right) = 1$ and $l\left(g^{-1}(B_{0})\right) = \mu_{f}\left(f^{-1}(g^{-1}(B_{0}))\right) = \mu_{f}(B_{0}) = 0$. 
Thus we see that $\mu_{g}$ is singular.

We assume that $\mu_{g}$ is singular.
Then, we see that $\mu_{f}$ is singular in the same manner as in the above argument.
\end{proof}

The following theorem gives a necessary and sufficient condition for the regularity of the stationary measure in this setting.
\begin{Thm}
Let the conditions $(i)$ and $(ii)$ as in Theorem 1.2 and $\nu$ be a unique $\mu$-stationary measure.
Then, we have \\
$(1)$ $\nu$ is absolutely continuous if and only if both $(i)$ and $(ii)$ hold. \\
$(2)$ $\nu$ is singular  if and only if either $(i)$ or $(ii)$ fails.
\end{Thm} 

\begin{proof}
It is sufficient to show ``if" parts.

(1) By noting Theorem 1.2(1),  we have that $f(x) = x/(-2c_{0}x+2c_{0}+1)$ and then $g(y) = (2c_{0}+1)y/(2c_{0}y+1)$.
By Lemma 4.2(2),  we have that $\mu_{g}$ is absolutely continuous and obtain the assertion (1).

(2)  We see that $\mu_{g}(f(I_{k}(x))) = \mu_{g}(g^{-1}(I_{k}(x))) = 2^{-k}$, $x \in [0,1)$, $k \geq 1$.
By Lemma 4.1(1), \\
$\mu_{g}\left(f(I_{k}(x))\right)  = \dfrac{1}{2} \left(\mu_{g}\left(A_{0}^{-1}(f(I_{k}(x)))\right) + \mu_{g}\left(A_{1}^{-1}(f(I_{k}(x)))\right)\right)$, $x \in [0,1)$, $k \geq 1$.
Then we see that (4.1) holds for $[a, b)$, $a, b \in \mathcal{C}$ and that $\mu_{g}$ is a $\mu$-stationary measure.
By noting Theorem 1.2(2), we have that $\mu_{f}$ is singular.
By Lemma 4.2(2),  we have that $\mu_{g}$ is singular and obtain the assertion (2).
\end{proof}


\section{Examples and Remarks}
The following example concerns Lebesgue's singular functions. 
\begin{Exa}
Let us define $2 \times 2$ real matrices $A_{p,0},  A_{p,1}$, $p \in (0,1)$,   by 
\[  A_{p,0} = \begin{pmatrix} p & 0  \\ 0 & 1 \end{pmatrix}, \, \,  A_{p,1} = \begin{pmatrix} 1 - p & p\\ 0 & 1 \end{pmatrix}.\] 
Then, $(A_{0}, A_{1}) = (A_{p,0}, A_{p,1})$ satisfies the conditions (A1)-(A3). 

Let $f_{p}$ be the solution of (1.1) for $(A_{0}, A_{1}) = (A_{p,0}, A_{p,1})$.
By the main theorems,  
we immediately have the following.\\
(1) $\mu_{f_{p}}$ is absolutely continuous  if $p = 1/2$, and $\mu_{f_{p}}$ is singular if $p \neq 1/2$. \\
(2) There exists a Borel set $K_{p}$ such that $\mu_{f_{p}}(K_{p}) = 1$ and $\dim_{H}(K_{p}) \leq s(p) / \log 2$. \\
(3) $\mu_{f_{p}}(K) = 0$ for any Borel set $K$ with $\dim_{H}(K) < s(p) / \log 2$.
\end{Exa}

The following example concerns the range of self-interacting walks on an interval in the author \cite{O1}.
\begin{Exa}
Let $x_{u} = 2/(1+\sqrt{1+8u^{2}})$, $u \geq 0$.
Let $\tilde A_{u, i}$, $i=0,1$, be two $2 \times 2$ real matrices given by 
\[ \tilde A_{u, 0} = \begin{pmatrix} x_{u} & 0  \\ -u^{2}x_{u}^{2} & 1 \end{pmatrix}, \, \tilde A_{u, 1} = \begin{pmatrix} 0 & x_{u} \\ -u^{2}x_{u}^{2} & 1 - u^{2}x_{u}^{2} \end{pmatrix}, \, u \geq 0. \]

Let $0 < u < \sqrt{3}$.
Then $(A_{0}, A_{1}) = (\tilde A_{u,0}, \tilde A_{u,1})$ satisfies the conditions (A1)-(A3). 
Let $g_{u}$ be the  solution of (1.1) for $(A_{0}, A_{1}) = (\tilde A_{u,0}, \tilde A_{u,1})$.
We remark that $\gamma = (1-u^{2}x_{u}^{2})/x_{u} = (1+x_{u})/2x_{u}$. 
By the definition of $x_{u}$, we see that each of the conditions  in Theorem 1.2 is equivalent to $x_{u} \neq 1/2$, that is, $u \neq 1$.
Then, by Theorem 1.2, we have that $\mu_{g_{u}}$ is singular for $0 < u < \sqrt{3}$ and $u \neq 1$,
and absolutely continuous for $u = 1$.

Let $0 < u < 1$.
Then we have that $x_{u} > 1/2$, $\alpha = \min \{0, -1/2, -u^{2}x_{u}\} = -1/2$, $\beta = 0$ and $\gamma < 3/2$.
Hence we see that $\gamma-2<\alpha$, in particular, $\gamma-2 \notin [\alpha, \beta]$.
By Theorem 1.1, we see that there exists a Borel set $\tilde K_{u}$ such that $\dim_{H}(\tilde K_{u}) \leq s(p_{0}(\alpha))/\log 2 = s(x_{u})/\log 2$ and $\mu_{g_{u}}(\tilde K_{u}) = 1$ 
and that $\mu_{g_{u}}(K) = 0$ for any Borel set $K$ with $\dim_{H}(K) < s(p_{0}(\beta))/\log 2 = s(2x_{u}/(1+x_{u}))/\log 2$.
\end{Exa}

\begin{Rem}
(1) Pincus \cite{P1}, \cite{P2} obtained results similar to Theorem 4.3. 
Hata \cite{H} Corollary 7.4 showed the singularity of the solution of (1.1) under the assumptions similar to the ones in \cite{P2} Theorem 2.1.\\
(2) Let $T : [0,1) \to [0,1)$ be given by $T(x) = 2x \mod 1$. 
Then, by computation, 
\[ \mu_{f} \left(T^{-1} (A)\right) = \int_{A} \left(\frac{d \Phi(A_{0}; \cdot)}{dz}\left(f(y)\right) + \frac{d \Phi(A_{1}; \cdot)}{dz}\left(f(y)\right)\right) \mu_{f}(dy), \, A \in \mathcal{B}([0,1)).\]
We see that $T$ is a non-singular transformation  on $[0,1)$ with respect to $\mu_{f}$, that is, $\mu_{f} \circ T^{-1} \ll \mu_{f}$ and $\mu_{f} \ll \mu_{f} \circ T^{-1}$.
We remark that $\mu_{f}$ is \textit{not} invariant with respect to $T$ in some cases.

\end{Rem}


\begin{thebibliography}{9}
\bibitem{A} K. Azuma, Weighted sums of certain random variables, T\^{o}hoku Math. J. (2) 19 (1967) 357-367.
\bibitem{BK} L. Berg and M. Kr\"uppel, De Rham's singular function and related functions, Z. Anal. Anwendungen 19 (2000), 227-237.  
\bibitem {de} G. de Rham, Sur quelques  courbes d\'{e}finies par des \'{e}quations fonctionalles, Rend. Sem. Mat. Torino 16 (1957), 101-113. 
\bibitem{Fu} H. Furstenberg, Noncommuting random products, Trans. Amer. Math. Soc. 108 (1963), 377-428.  
\bibitem {H} M. Hata, On the structure of self-similar sets, Japan J. Appl. Math. 2 (1985), 381-414.
\bibitem{Ka} K. Kawamura, On the set of points where Lebesgue's singular function has the derivative zero, Proc. Japan Acad. Ser. A Math Sci. 87 (2011), 162-166.
\bibitem{Kr} M. Kr\"uppel, De Rham's singular function, its partial derivatives with respect to the parameter and binary digit sums, Rostock. Math. Kolloq. 64 (2009), 57-74.
\bibitem {O1} K. Okamura, On the range of  self-interacting random walks on an interval, preprint, available at arXiv 1207.1245v1. 
\bibitem{P1} S. Pincus, A class of Bernoulli random matrices with continuous singular stationary measures, Ann. Prob. 11 (1983) 931-938. 
\bibitem{P2} S. Pincus, Singular stationary measures are not always fractal, J. Theor. Prob. 7 (1994) 199-208.   
\bibitem{Pr} V. Y. Protasov, On the regularity of de Rham curves, Izv. Math. 68 (2004) 567-606. 
\end{thebibliography}
\end{document}